\documentclass[11pt, leqno, twoside]{article}

\usepackage{amssymb}
\usepackage{amsmath}
\usepackage{amsthm}
\usepackage{amsfonts}
\usepackage{mathrsfs}
\usepackage{indentfirst}

\usepackage{color}

\allowdisplaybreaks

\usepackage{txfonts}
\def\rr{{\mathbb R}}
\def\rn{{{\rr}^n}}

\def\cc{{\mathbb C}}
\def\nn{{\mathbb N}}
\def\cp{{\mathcal P}}

\def\cm{{\mathcal M}}

\def\ca{{\mathcal A}}
\def\cd{{\mathcal D}}
\def\cb{{\mathcal B}}
\def\cw{{\mathcal W}}

\def\fz{\infty}

\def\lz{\lambda}

\def\lf{\left}
\def\r{\right}

\def\ls{\lesssim}

\def\wz{\widetilde}

\def\cx{{\mathcal X}}

\def\loc{{\mathop\mathrm{\,loc\,}}}
\def\esinf{\mathop\mathrm{\,ess\,inf\,}}
\def\esup{\mathop\mathrm{\,ess\,sup\,}}

\def\q1{\wz q}
\def\Q1{q_1}
\def\vlp{{L^{p(\cdot)}(\cx)}}

\def\aaa{{\ca_{p(\cdot)}^s(\cx)}}
\def\bbb{{\cb_{p(\cdot)}^s(\cx)}}
\def\www{{\cw_{p(\cdot)}^s(\cx)}}

\def\loc{{\mathop\mathrm{loc\,}}}

\newtheorem{thm}{Theorem}[section]
\newtheorem{prop}[thm]{Proposition}
\newtheorem{lem}[thm]{Lemma}
\newtheorem{cor}[thm]{Corollary}
\theoremstyle{definition}
\newtheorem{defn}[thm]{Definition}
\newtheorem{rem}[thm]{Remark}

\numberwithin{equation}{section}
\numberwithin{equation}{section}

\begin{document}
\title{\bf\Large  Characterizations of variable fractional Haj{\l}asz-Sobolev spaces
\footnotetext{\hspace{-0.35cm} 2020 {\it
Mathematics Subject Classification}. 46E36,42B25.
\endgraf {\it Key words and phrases}. Sobolev space, variable exponent,
Haj{\l}asz gradient, metric measure space.
\endgraf This project is supported by the
Natural Science Foundation of Changsha
(Grant No. kq2202237).}}
\author{Xiaosi Zhang\footnote{Corresponding
author,\ {}} \  and Qi Sun}
\date{ }
\maketitle

\vspace{-0.7cm}
\begin{center}
\begin{minipage}{13cm}
{\small {\bf Abstract}\quad}
Let $(\mathcal X,\rho,\mu)$ be a space of homogeneous type,
$p(\cdot):\ \mathcal X \to[1,\infty)$ a variable exponent satisfying
the globally log-H\"older continuous condition, and $s\in(0,\infty)$.
In this article,
the author introduce the variable fractional Sobolev spaces
$W_{p(\cdot)}^s(\mathcal X)$ on $\mathcal X$ via Haj{\l}asz gradient.
Using various maximal functions, several
characterizations of this space are established.

\end{minipage}
\end{center}
\section{Introduction}

The purpose of this note is to study the variable fractional
Haj{\l}asz-Sobolev spaces $\www$ on a space of homogeneous type $(\cx,\rho,\mu)$.
We mainly establish equivalent characterizations of the variable fractional
Haj{\l}asz-Sobolev spaces via maximal functions.

Recall that the well-known Haj{\l}asz-Sobolev space $\cw_p^1(\cx)$ on metric measure spaces
was originally developed by Haj{\l}asz \cite{Ha96} by introducing
the Haj{\l}asz gradient. A non-negative function $g$ is called a Haj{\l}asz gradient of
$f\in L^p(\cx)$, if it holds that, for almost every $x\in\cx$,
$$|f(x)-f(y)|\leq \rho(x,y)[g(x)+g(y)].$$

Due to the remarkable paper of Haj{\l}asz \cite{Ha96}, Sobolev spaces on metric measure
spaces attract much attention. Indeed, via giving a new characterization
of the classical Sobolev on $\rn$ via Haj{\l}asz gradient in \cite{Ha96}, which
only depends on the metric and measure,
the author hence defined the 1-order Sobolev spaces on metric measure spaces.
This opened the door for the study of Sobolev spaces
on metric measure spaces.
Motivated by \cite{Ha96} and as generalization,
Liu in \cite{ll02} investigated the Sobolev spaces with higher order on on metric measure
spaces and mainly gave their equivalent characterizations in terms of polynomials.
Later, Yang \cite{Ya03} studied the fractional
Haj{\l}asz Sobolev space on the setting of the space of homogeneous type
and established some equivalent characterizations using (sharp) maximal functions,
which was a generalization of the corresponding results on subsets of $\rn$
in \cite{hu03}.
 Moreover,
the generalization of Sobolev spaces on metric measure spaces into the setting of
variable exponents is a natural subject.
For example, in \cite{hh06}, Harjulehto et al. sdudied variable exponent Sobolev
spaces on metric measure spaces employing two definitions: Haj{\l}asz type and
Newtonian type. We also refer the reader to references \cite{as09,dh11, ha18}.

As a generalization of classical Lebesgue spaces $L^p$ with a constant exponent $p\in(0,\fz)$,
the Lebesgue spaces with variable exponent $p(\cdot)$ on the
Euclidean space $\rn$ as well as on metric measure space attracts more and more attention
in recent years. Indeed, the variable Lebesgue space with variable exponent
have been independently investigated in \cite{fz01,kr91,S00},
and become the intensive research due to the development in \cite{cruz03} of Cruz-Uribe and
\cite{din04} of Diening.
This kind of variable Lebesgue spaces was further used into the study on
the variable Hardy space \cite{ns12,cw14,zsy16}, the variable Besov and Triebel-Lizorkin spaces
\cite{ah10,dhr09,yzy151,yzy15} and also the variable Sobolev space \cite{hh06,dh11}.
The variable Lebesgue space $\vlp$ on a metric measure space $(\cx,\rho,\mu)$ is
defined to be the set of all
measurable functions $f$ satisfying $\int_\cx [|f(x)|/\lambda]^{p(x)}\,d\mu(x)<\fz$
for some $\lambda\in(0,\fz)$,
which obviously is a generalization of the classical Lebesgue space $L^p(\cx)$
 via taking $p(\cdot)\equiv p\in[1,\fz)$.
They are widely used in harmonic analysis
\cite{cf13}, in partial differential equations and variation calculus \cite{am05,hhl08,su09},
in fluid dynamics and image processing \cite{am02}.
For the history of variable functions on metric measure space, we also refer to
\cite{as09,ah15,hh04,hh06,ha18} and their references.
Particularly, Harjulehto et al. \cite{hh04} proved that the boundedness of Hardy-Littlewood maximal
operator on metric measure space but limited on bounded spaces.
This was further
generalized by Adamowicz et al. \cite{ah15} to unbounded metric measure space,
which makes it more widely to investigate the variable function spaces on
metric measure space.

The spaces of homogeneous type was introduced in \cite[Chapter 3]{cw71}.

\begin{defn}\label{de-1}
Let $(\cx, \rho, \mu)$ be a (quasi-)metric measure space, where $\mu$
is a Borel regular measure such that
all balls have positive measures defined by $\mu$
and $\rho$ is quasi-metric, namely, there exists a constant $A \in [1,\fz)$ such that,
for all $x, y$ and $z \in \cx$,
\begin{equation}\label{eq-2}
\rho(x,y)\leq A[\rho(x,z)+\rho(z,y)].
\end{equation}
The triple $(\cx,\rho,\mu)$ is called a space of
homogeneous type measure if the measure $\mu$ is doubling, namely,
there is a constant $C_\mu \geq 1 $ such
that, for all $0 < r < diam\cx$ and all $x\in \cx$,
\begin{equation*}
\mu(B(x,2r))\leq C_\mu \mu(B(x,r)).\label{eq-1}
\end{equation*}
The constant $C_\mu$ is called the doubling constant of $\mu$.
\end{defn}

\begin{rem}
From \eqref{eq-1}, it is obviously that for any $\alpha>0$,
\begin{equation}\label{eq-16}
\mu(B(x,\alpha r))\leq C_\alpha \mu(B(x,r)),
\end{equation}
where $C_\alpha=(2\alpha)^{\log_2{C_\mu}}$ depends on $C_\mu$ and $\alpha$.
\end{rem}

We next list some basic notions and notation for the variable function space.
Denote by $\cp(\cx)$ the set of all variable
exponents $p(\cdot):\ \cx \to[1,\fz)$ satisfying
$$1\leq p^-:=\esinf_{x\in\cx}p(x)\le p^+:=\esup_{x\in\cx}p(x)<\fz.$$
For any $p(\cdot)\in \cp(\cx)$,
the \emph{variable Lebesgue space} $\vlp$ is defined to be the set of all
$\mu$-measurable functions $f$ such that
$$\|f\|_{\vlp}:=\inf\lf\{\lambda\in(0,\fz):\ \int_\cx\lf[\frac{|f(x)|}{\lz}\r]^{p(x)}
\,d\mu(x)\leq 1\r\}<\fz.$$

\begin{defn}
A function $p(\cdot):\ \cx\to\rr$ is said to be globally log-H\"older
continuous in $\cx$ with a basepoint $x_0\in \cx$, denoted by $p(\cdot)\in C^{\log}(\cx)$,
if there exists $C_{\log}(p)\in(0,\fz)$ such that, for all
$x,y\in \cx$,
\begin{equation*}
|p(x)-p(y)|\leq\frac{C_{\log}(p)}{\log (e +1/d(x,y))},
\end{equation*}
and there exist $p_\infty\in\rr$  and a constant $C_\fz(p)\in(0,\fz)$ such that,
for all $x\in \cx$,
\begin{equation*}
|p(x)-p_\infty|\leq \frac{C_\fz(p)}{\log (e +d(x,x_0))}.
\end{equation*}
\end{defn}

\begin{rem}\label{rem-1}
Let $\cx$ be a space of homogeneous type and $p(\cdot)\in P(\cx)$.
\begin{enumerate}
\item[{\rm(i)}] $\vlp$ is a Banach space (see \cite[Lemma 3.1]{hh04}).
\item[{\rm(ii)}] It is easy to see that, for any
$s\in[\frac{1}{p^-},\infty), \||f|^s\|_{\vlp}=\|f\|_{L^{sp(\cdot)}(\cx)}^s.$
\item[{\rm(iii)}]  If $p(\cdot)\in C^{\log}(\cx)$ with $p^-\in(1,\fz)$,
 then, for all $f\in\vlp$, it holds that
    $$\|\cm f\|_{\vlp}\leq C\|f\|_{\vlp},$$
where $C$ depends on $p(\cdot)$ (see \cite[Corollary 1.8]{ah15}).

\item[(iv)] If $\cx$ is bounded, $p(\cdot), q(\cdot)\in \cp(\cx)$ and
$p(x)\leq q(x)$, then
    $\|u\|_{\vlp}\leq\|u\|_{L^{q(\cdot)}(\cx)}$ (see \cite[Theorem 2.8]{kr91}).
\end{enumerate}
\end{rem}
Recall that the
Hardy-Littlewood maximal operator, denoted by $\cm$, is defined for a
locally integrable function $f$ by setting, for any $x\in\cx$,
$$\cm(f)(x):=\sup_{B\ni x}\frac{1}{\mu(B)}\int _{B}|f(y)|\,d\mu(y),$$
where the supremum is taken over all balls $B\subset\cx$.

We now give the definition of variable fractional Haj{\l}asz-Sobolev spaces
on the space of homogeneous type.
\begin{defn}\label{de-4}
Let $\cx$ be a space of homogeneous type, $p(\cdot)\in \cp(\cx)$ and
$s\in(0,\fz)$.

(i) Let $f\in\vlp$. Then a non-negative function $g\in\vlp$ is
called to be a $s$-order Haj{\l}asz gradient of $f$ if, for a.e. $x,\ y\in \cx$,
$
\lf|f(x)-f(y)\r|\leq \rho(x,y)^s\lf[g(x)+g(y)\r].
$
Denote by $\cd_s(f)$ the set of all $s$-order Haj{\l}asz gradients of $f$.

(ii) The \emph{variable fractional Haj{\l}asz-Sobolev space} $\www$ is defined
to be the set of all $f\in \vlp$ satisfying $\cd_s(f)\neq \emptyset$.
Moreover, for any $f\in \www$, we define
$$\|f\|_{\cw_{p(\cdot)}^s}(\cx)=\|f\|_\vlp+\inf_{g\in{\cd_s(f)}}\|g\|_\vlp,$$
where the infimum is taken over all function $g\in \cd_s(f)$.
\end{defn}

\begin{rem}
If $p(\cdot)\equiv p\in[1,\fz)$, then the space $\www$ goes back
to the space $\cw_{p}^s(\cx)$, which was studied by Yang \cite{Ya03}.
\end{rem}

One of the main results in this article is stated as follows.
\begin{thm}\label{thm-1}
Let $\cx$ be a space of homogeneous type,
$p(\cdot)\in C^{\log}(\cx)$ with $p^-\in(1,\fz)$ and $s\in(0,\infty)$.
Then $f\in \www$ if and only if $f\in \vlp$ and there exists
a function $h\in \vlp$ such that, for all balls $B\subset\cx$ and $\mu$-a.e. $x\in B$,
\begin{equation}
\lf|f(x)-\frac{1}{\mu(B)}\int_{B} f(z)\,d\mu(z)\r|\le r(B)^sh(x).
\label{eq-17}
\end{equation}
 Moreover, it holds true that
$$\|f\|_{\www}\sim \|f\|_{\vlp}+\inf_{h}\|h\|_{\vlp},$$
where the infimum is taken over all functions $h$ satisfying \eqref{eq-17} and the implicit
constants are independent of $f$.
\end{thm}

Moreover, we can also characterize the space $\www$ via the following sharp
maximal function, which is the second main result of this article.
For any $u,s\in(0,\fz)$, we define
the sharp maximal function by setting, for any $f\in L_{\loc}^1(\cx)$ and $x\in\cx$,
\begin{equation}\label{eq-7}
f_{u}^s(x):=\sup_{t\in(0,\fz)}t^{-s}
\lf\{\frac{1}{\mu(B(x,t))}\int_{B(x,t)}\lf|f(y)-\frac{1}{\mu(B(x,t))}\int_{B(x,t)}f(z)\,
d\mu(z)\r|^u\,d\mu(y)\r\}^{1/u}.
\end{equation}

\begin{thm}\label{thm-3}
Let $\cx$ be a space of homogeneous type, $p(\cdot)\in C^{\log}(\cx)$
with $p^-\in(1,\fz)$ and $s\in(0,\infty)$. Then $f\in\www$ if
and only if $f\in\vlp$ and $f_{u}^s\in\vlp$ with $u\in[1,p^-)$. Moreover, it holds
true that
$$\|f\|_\www\sim\|f\|_\vlp+\|f_{u}^s\|_\vlp$$
with the implicit constants being independent of $f$.
\end{thm}

Finally, when $\cx$ is a bounded space of homogeneous type,
we have the following characterization of the $\www$ via integrals.

\begin{thm}\label{thm-2}
Let $\cx$ be a space of homogeneous type satifying
$\mu(\cx)<\fz$, $p(\cdot)\in C^{\log}(\cx)$ with $p^-\in(1,\fz)$
and $s\in(0,\infty).$ Then $f\in \www$ if and only if $f\in \vlp$ and
there exist some constant $q\in[1,p^-)$ and a non-negative
function $\varphi \in \vlp$ such that, for every ball $B\subset\cx$ and a.e. $x\in B$,
\begin{equation}
\frac{1}{\mu(B)}\int_{B}\lf|f(x)-\frac{1}{\mu(B)}\int_{B} f(y)\,d\mu(y)\r|\,
d\mu(x)\le r(B)^s\lf(\frac{1}{\mu(B)}\int_{B}\varphi(x)^q\,d\mu(x)\r)^{1/q}.
\label{eq-18}
\end{equation}
Moreover, it holds true that
$$\|f\|_{\www}\sim \|f\|_{\vlp}+\inf_{\varphi}\|\varphi\|_{\vlp},$$
where the infimum is taken over all functions $\varphi$ satisfying \eqref{eq-18}
and the implicit constants are independent of $f$.
\end{thm}

Theorems \ref{thm-1}-\ref{thm-2} are proved in Section \ref{r-pf}.
In particular, Theorem \ref{thm-2} is the integral form of Theorem \ref{thm-1},
for which we first establish the relationship between the two results
 (see Lemma \ref{lem-3} below) and then, by using the equivalence of them on a bounded
 space of homogeneous type, prove Theorem \ref{thm-2}.
 The key tools used to prove them are the boundedness of the maximal
operator and lebesgue differentiation theorem.

 Throughout this article, let $\nn:=\{1,2,\dots\}$ and
${\rm diam}\ \cx:= \sup\{\rho(x,y):x,y\in\cx\}$. For any $a\in(0,\fz)$ and ball
$B(x,r)\subset\cx$, let $aB(x,r):=\{y\in\cx:\ \rho(x,y)<ar\}$.
We denote by $C$ a positive constant which is independent of the main parameters,
but it may vary from line to line. We denote by $f\lesssim g$ (resp.,$f\gtrsim g$)
if $f\le Cg$ (resp., $f\ge Cg$) for a positive constant $C$, and $f\sim  g$
amounts to $f\gtrsim g\gtrsim f$.

\section{ Proofs of the main results\label{r-pf}}
In this section, we prove our main results and begin with the proof of Theorem
\ref{thm-1}. To this end, for any $p(\cdot)\in \cp(\cx)$ and $s\in(0,\fz)$, let
$$\bbb:=\lf\{f\in\vlp:\ \|f\|_\bbb=\|f\|_\vlp+\inf_{h}\|h\|_\vlp<\fz \r\},$$
where the infimum is taken over all functions $h$ satisfying that, for all balls
$B\subset\cx$ and a.e. $x\in B$,
\begin{equation}\label{eq-2.0x}
\lf|f(x)-\frac{1}{\mu(B)}\int_{B} f(z)\,d\mu(z)\r|\le r(B)^sh(x).
\end{equation}

\begin{proof}[Proof of Theorem \ref{thm-1}]
Let $f\in \www$. Then by definition, we know that, for any given $\varepsilon >0$, there is a non-negative
function $g\in \vlp$ such that for a.e. $x,y\in \cx$,
\begin{equation}\label{eq-2.1x}
|f(x)-f(y)|\leq
\rho(x,y)^s[g(x)+g(y)]
\end{equation}
and
\begin{equation}\label{eq-2.2x}
||f||_\vlp+||g||_\vlp<||f||_\www+\varepsilon.
\end{equation}
By \eqref{eq-2.1x}, we deduce that, for any ball $B\subset\cx$ and a.e. $x\in B$,
\begin{align*}
\lf|f(x)-\frac{1}{\mu(B)}\int_Bf(y)\,d\mu(y)\r|
&\leq \frac{1}{\mu(B)}\int_B \lf|f(x)-f(y)\r|\,d\mu(y)\\
&\leq 2^sr(B)^s\frac{1}{\mu(B)}\int_B\lf[g(x)+g(y)\r]\,d\mu(y)\\
&\leq 2^sr(B)^s\lf[g(x)+\cm (g)(x)\r]\\
&\leq 2^{s+1}r(B)^s \cm (g)(x).
\end{align*}
Therefore, $f\in\bbb$, and moreover, by Remark \ref{rem-1} (iii) and \eqref{eq-2.2x},
$$\|f\|_\bbb \lesssim \|f\|_\vlp+\|\cm (g)\|_\vlp
\lesssim \|f\|_\vlp+\|g\|_\vlp
\lesssim \|f\|_\www+\varepsilon.$$
Letting $\varepsilon\to 0$, we then conclude that $\|f\|_\bbb \lesssim \|f\|_\www.$

Conversely, let $f\in \bbb$. Then, for any given $\varepsilon>0$, there is
a non-negative function $h\in \vlp$ satisfying \eqref{eq-2.0x}
and
\begin{equation*}
\|f\|_\vlp+\|h\|_\vlp<\|f\|_\bbb+\varepsilon.
\end{equation*}
Moreover, we have, for any ball $B\subset\cx$ and a.e. $x,y\in B$,
$$\lf|f(x)-f(y)\r|\leq r(B)^s\lf[h(x)+h(y)\r],$$
which, together with choosing  suitable $B$, implies that for a.e. $x,y\in\cx$,
$$\lf|f(x)-f(y)\r|\leq \rho(x,y)^s\lf[h(x)+h(y)\r].$$
By this, we find that $f\in\www$ and
$$\|f\|_\www \leq \|f\|_\vlp+\|h\|_\vlp <\|f\|_\bbb+\varepsilon.$$
Letting $\varepsilon\to 0$, we conclude that $\|f\|_\www\leq\|f\|_\bbb.$
This completes the proof of Theorem \ref{thm-1}.
\end{proof}

Next we will prove the second main result by using Theorem \ref{thm-1}.
\begin{proof}[Proof of Theorem \ref{thm-3}]
Let $f\in\www$. By Theorem \ref{thm-1}, we find that for any given
$\varepsilon>0$, there exists a function $h\in\vlp$ satisfying \eqref{eq-2.0x}
and
\begin{equation}\label{eq-2.5s}
 \|f\|_\vlp+\|h\|_\vlp<\|f\|_\bbb+\varepsilon.
\end{equation}
By \eqref{eq-2.0x}, we deduce that, for any $x\in \cx$,
\begin{align*}
f_{u}^s(x)&=\sup_{t\in(0,\fz)}t^{-s}
\lf\{\frac{1}{\mu(B(x,t)}\int_{B(x,t)}
\lf|f(y)-\frac{1}{\mu(B(x,t))}\int_{B(x,t)}f(z)\,d\mu(z)\r|^u\,d\mu(y)\r\}^
{1/u}\\
&\leq \sup_{t\in(0,\fz)}\lf[\frac{1}{\mu(B(x,t))}\int_{B(x,t)}h(y)^u\,d\mu(y)\r]^{1/u} \\
&\leq [\cm(h^u)(x)]^{1/u},
\end{align*}
which, together with Remark \ref{rem-1} (iii) and \eqref{eq-2.5s}, implies that
\begin{align*}
\|f\|_\vlp+\|f_{u}^s\|_\vlp &\ls \|f\|_\vlp+\|[\cm(h^u)]^{1/u}\|_\vlp\\
&\lesssim \|f\|_\vlp+\|h\|_\vlp\\
&\lesssim\|f\|_\bbb+\varepsilon.
\end{align*}
Letting $\varepsilon\to0$, we get
$$\|f\|_\vlp+\|f_{u}^s\|_\vlp \lesssim\|f\|_\bbb\lesssim\|f\|_\www.$$

On the other hand, let $f\in\vlp$ and $f_{u}^s\in\vlp$.
Let $x$ be a given Lebesgue point of $f$. For any given $t\in(0,\fz)$, let
$y\in B(x,t/(2A))$ be a Lebesgue point of the function
$$f(v)-\frac{1}{\mu(B(x,t))}\int_{B(x,t)}
f(z)\,d\mu(z),$$
where $A$ is the constant mentioned in \eqref{eq-2}.
Denote by $B_0:=B(x,t)$ and $B_j:=B(y,2^{-j}t/A)$
for $j\in \nn$. Then by the Lebesgue differentiation theorem,
we obtain
\begin{align*}
\lf|f(y)-\frac{1}{\mu(B(x,t))}\int_{B(x,t)}f(z)\,d\mu(z)\r|&
=\lim_{j\to\infty}\frac{1}{\mu(B_j)}\int_{B_j}\lf|f(v)-\frac{1}{\mu(B(x,t))}\int_{B(x,t)}
f(z)\,d\mu(z)\r|\,d\mu(v)\\
&\leq \varlimsup_{j\to\infty}\frac{1}{\mu(B_j)}\int_{B_j}\lf|f(v)-\frac{1}{\mu(B_j)}\int_{B_j} f(z)\,d\mu(z)\r|\,d\mu(v)\\
&\quad+\varlimsup_{j\to\infty}\lf|\frac{1}{\mu(B_j)}\int_{B_j}f(z)\,
d\mu(z)-\frac{1}{\mu(B_0)}\int_{B_0}f(z)\,d\mu(z)\r|\\
&= J_1+J_2.
\end{align*}
For $J_1$, it follows from the H\"older inequality that
\begin{align*}
J_1 &=\varlimsup_{j\to\infty}\frac{1}{\mu(B_j)}\int_{B_j}\lf|f(v)-\frac{1}
{\mu(B_j)}\int_{B_j}
f(z)\,d\mu(z)\r|\,d\mu(v)\\
&\leq \varlimsup_{j\to\infty}\lf\{\frac{1}{\mu(B_j)}\int_{B_j}\lf|f(v)-\frac{1}
{\mu(B_j)}\int_{B_j}
f(z)\,d\mu(z)\r|^u\,d\mu(v)\r\}^{1/u}\\
&\leq \varlimsup_{j\to\infty}2^{-js}A^{-s}t^sf_{u}^s(y)
=0.
\end{align*}
For $J_2$, we first have $B_1\subset
B_0\subset 4A^2B_1$, which together with \eqref{eq-16} implies $\mu(B_0)\sim \mu(B_1)$.
Thus, by the H\"older inequality, we find that
\begin{align*}
J_2&\leq \varlimsup_{j\to\infty} \sum_{l=1}^{j-1}\lf|\frac{1}{\mu(B_{l+1})}
\int_{B_{l+1}}f(z)\,d\mu(z)-\frac{1}{\mu(B_{l})}\int_{B_l}
f(z)\,d\mu(z)\r|\\
&\quad\quad+\lf|\frac{1}{\mu(B_{1})}\int_{B_1}f(z)\,d\mu(z)-\frac{1}
{\mu(B_{0})}\int_{B_0} f(z)\,d\mu(z)\r|\\
&\lesssim \sum_{l=1}^{\infty}\frac{1}{\mu(B_{l})}\int_{B_l}\lf
|f(v)-\frac{1}{\mu(B_{l})}\int_{B_l}
f(z)\,d\mu(z)\r|\,d\mu(v)\\
&\quad\quad+\frac{1}{\mu(B_{0})}\int_{B_0}
\lf|f(v)-\frac{1}{\mu(B_{0})}\int_{B_0} f(z)\,d\mu(z)\r|\,d\mu(v)\\
&\lesssim \sum_{l=1}^{\infty} \lf\{\frac{1}{\mu(B_{l})}\int_{B_l}\lf|f(v)-\frac{1}{\mu(B_{l})}\int_{B_l} f(z)\,
d\mu(z)\r|^u\,d\mu(v)\r\}^\frac{1}{u}\\
&\quad\quad +\lf\{\frac{1}{\mu(B_{0})}\int_{B_0}
\lf|f(v)-\frac{1}{\mu(B_{0})}\int_{B_0}f(z)\,d\mu(z)\r|^u\,d\mu(v)\r\}^\frac{1}{u}\\
&\lesssim \sum_{l=1}^{\infty}2^{-ls}t^sA^{-s} f_{u}^s(y)+t^sf_{u}^s(x)
\ls t^s\lf[f_{u}^s(y)+f_{u}^s(x)\r].
\end{align*}
Therefore, we have
$$\lf|f(y)-f(x)\r|\lesssim t^s\lf[f_{u}^s(y)+f_{u}^s(x)\r],\quad a.e.\,y\in B(x,t/2A).$$
Since for any $x,\ y\in \cx$ there exists $t\in(0,\fz)$ such that $y\in B(x,t/(2A))$,
it follows that
$$\lf|f(y)-f(x)\r|\lesssim \rho(x,y)^s\lf[f_{u}^s(y)+f_{u}^s(x)\r].$$
By this, we conclude that $f\in\www$ and $\|f\|_\www\lesssim\|f\|_\vlp+\|f_{u}^s\|_\vlp.$
This completes the proof of Theorem \ref{thm-3}.
\end{proof}

To prove Theorem \ref{thm-2}, we use the following notion. For any
$p(\cdot)\in \cp(\cx)$ and $s\in(0,\fz)$,
$$\aaa:=\lf\{f\in\vlp:\ \|f\|_\bbb=\|f\|_\vlp+\inf_{\varphi}\|\varphi\|_\vlp<\fz \r\},$$
where the infimum is taken over all functions $\varphi$ satisfying that for every ball $B\subset\cx$ and $q\in[1,p^-)$,
\begin{equation}\label{eq-2.5x}
\frac{1}{\mu(B)}\int_{B}\lf|f(x)-\frac{1}{\mu(B)}\int_{B} f(y)\,d\mu(y)\r|\,
d\mu(x)\le r(B)^s\lf(\frac{1}{\mu(B)}\int_{B}\varphi(x)^q\,d\mu(x)\r)^{1/q}.
\end{equation}
Moreover, we need two technical lemmas. The following one comes from
\cite[Theorem 1.1]{Ya03}.

\begin{lem}\label{lem-2}
Let $q\in[1,\fz)$ and $s\in(0,\fz)$. If $f\in L_{\loc}(\cx)$
and there are a non-negative function $g\in L^q(\cx)$ and some $\lambda\in[1,\fz)$
such that, for every ball $B\subset \cx$ and a.e. $x\in B$, the
Poincar\'{e} inequality
\begin{equation*}
\frac{1}{\mu(B)}\int_B\lf|f(x)-\frac{1}{\mu(B)}\int_{B}f(z)\,d\mu(z)\r|\,d\mu(x)\leq
Cr(B)^s\lf(\frac{1}{\mu(\lambda B)}\int_{\lambda
B}g(x)^q\,d\mu(x)\r)^{1/q}
\end{equation*}
holds true. Then for any ball $B\subset \cx$ and a.e. $x\in B$,
$$\lf|f(x)-\frac{1}{\mu(B)}\int_{B}f(z)\,d\mu(z)\r|\leq Cr(B)^s[\cm(g^q)(x)]^{1/q},$$
where $C$ is independent of $x$ and $B$.
\end{lem}

Using Lemma \ref{lem-2}, we can establish the relation between the space
$\aaa$ and $\bbb$ on the bounded space of homogeneous type as follows.

\begin{lem}\label{lem-3}
Let $\cx$ be a space of homogeneous type satisfying
$\mu(\cx)<\fz$, $p(\cdot)\in C^{\log}(\cx)$  with $p^-\in(1,\fz)$ and $s\in(0,\fz)$.
Then $f\in\aaa$ if and only if $f\in\bbb$. Moreover, it holds true that
$\|f\|_{\aaa}\sim\|f\|_{\bbb},$
where the implicit constants are independent of $f$.
\end{lem}
\begin{proof}
Let $f\in \aaa $. Then for any given $\varepsilon >0$, there exist
a constant $q\in[1,p^-)$ and a non-negative function $\varphi\in \vlp$ satisfying
\eqref{eq-2.5x} and
$$\|f\|_{\vlp}+\|\varphi\|_{\vlp}<\|f\|_{\aaa}+\varepsilon.$$
By Remark \ref{rem-1} (iv) and Lemma \ref{lem-2}, we deduce that, for all balls $B\subset \cx$
and a.e. $x\in B$,
$$\lf|f(x)-\frac{1}{\mu(B)}\int_Bf(z)\,d\mu(z)\r|\leq Cr(B)^s\cm(\varphi^q)(x)^{1/q},$$
which, together with Remark \ref{rem-1} (ii) and (iii), implies that $f\in \bbb$ and
\begin{align*}
\|f\|_\bbb &\lesssim \|f\|_\vlp+\|[\cm(\varphi^q)]^{1/q}\|_\vlp \\
 &\lesssim \|f\|_\vlp+\|\varphi\|_\vlp\\
 & \lesssim\|f\|_\aaa +\varepsilon.
\end{align*}
Letting $\varepsilon\to 0$, we conclude that
$\|f\|_\bbb\lesssim\|f\|_\aaa.$

Conversely, let $f\in \bbb$. Then for any $\varepsilon>0$, there exists a function
$h\in \vlp$ satisfying \eqref{eq-2.0x} and
\begin{equation}
\|f\|_\vlp+\|h\|_\vlp<\|f\|_\bbb+\varepsilon.
\label{eq-6}
\end{equation}
By \eqref{eq-2.0x} and the H\"older inequality, we deduce that for any ball
$B\subset \cx$ and a.e. $x\in B$,
\begin{align*}
\frac{1}{\mu(B)}\int_B\lf|f(x)-\frac{1}{\mu(B)}\int_Bf(y)\,d\mu(y)\r|\,d\mu(x)
&\leq r(B)^s\frac{1}{\mu(B)}\int_B h(x)\,d\mu(x)\\
&\leq r(B)^s\lf\{\frac{1}{\mu(B)}\int_Bh(x)^q\,d\mu(x)\r\}^{1/q},
\end{align*}
where the $q\in[1,p^-)$. Thus, together with \eqref{eq-6}, we have $f\in \aaa$ and
$$
\|f\|_\aaa\leq\|f\|_\vlp+\|h\|_\vlp
<\|f\|_\bbb+\varepsilon.
$$
Letting $\varepsilon\to 0$, we conclude that $\|f\|_\aaa\leq\|f\|_\bbb.$
This finishes the proof of Lemma \ref{lem-3}.
\end{proof}

\begin{proof}[Proof of Theorem \ref{thm-2}]
The conclusion of Theorem \ref{thm-2} is just
a consequence of Theorem \ref{thm-1} and Lemma \ref{lem-3}.
Indeed, we easily find that
$$\|f\|_\vlp+\inf_{\varphi}\|\varphi\|_\vlp\sim\|f\|_{\aaa}\sim
\|f\|_{\bbb}\sim\|f\|_{\vlp}+\inf_{h}\|h\|_{\vlp}\sim\|f\|_{\www},$$
where the infimums are taken over all
non-negative function $\varphi\in\vlp$ satisfying \eqref{eq-2.5x}.
The proof of Theorem \ref{thm-2} is completed.
\end{proof}

Next we introduce some variants of the sharp maximal function in \eqref{eq-7} as follows.
\begin{equation*}
 \widetilde f_{u}^s(x):=\sup_{t\in(0,\fz)}t^{-s}\inf_{B\in\cc}
 \lf\{\frac{1}{\mu(B(x,t))}\int_{B(x,t)}|f(y)-B|^u\,d\mu(y)\r\}^{1/u}
\end{equation*}
and
\begin{equation*}
\overline f_{u}^s(x):=\sup_{t\in(0,\fz)}t^{-s}
\lf\{\frac{1}{\mu(B(x,t))}\int_{B(x,t)}|f(y)-f(x)|^u\,d\mu(y)\r\}^{1/u},
\end{equation*}
where the function $f$ and the constants $s,u$ are as in \eqref{eq-7}.

\begin{rem}\label{rem-3}
For $s\in \rr, u \in [1,\infty)$ and $x \in \cx$. Let $f$ be a locally
integrable function, then
\begin{equation*}
\widetilde f_{u}^s(x)\leq f_{u}(x)\leq2\widetilde f_{u}^s(x)
\end{equation*}
and
\begin{equation*}
\widetilde f_{u}^s(x)\leq \overline f_{u}^s(x)\leq f_{u}^s(x)+2\cm (f)(x).
\end{equation*}
For the proof, we refer to \cite{Ya03}.
\end{rem}

Theorem \ref{thm-3} and Remark \ref{rem-3} yield the following results.
\begin{cor}\label{co-1}
Let $\cx$ be a space of homogeneous type, $p(\cdot)\in C^{\log}(\cx)$ with
$p^-\in(1,\fz)$, $u\in[1,p^-)$ and $s\in(0,\fz).$ Then the following
statements are equivalent:
\begin{enumerate}
\item[{\rm(i)}] $f\in\www$;

\item[{\rm(ii)}] $f\in\vlp$ and $f_{u}^s\in\vlp$;

\item[{\rm(iii)}] $f\in\vlp$ and
$\widetilde f_{u}^s\in\vlp$;

\item[{\rm(iv)}] $f\in\vlp$  and
$\overline f_{u}^s\in\vlp$.
\end{enumerate}
Moreover, it holds true
\begin{align*}
\|f\|_\www\thicksim &\|f\|_\vlp+\|f_{u}^s\|_\vlp\\
\thicksim&\|f\|_\vlp+\|\widetilde f_{u}^s\|_\vlp\\
\thicksim&\|f\|_\vlp+\|\overline f_{u}^s\|_\vlp,
\end{align*}
where the implicit constants are independent of $f$.
\end{cor}

We end this section by the following conclusion, which has its own interest.
\begin{prop}\label{prop-1}
The space $\www$  is a Banach space.
\end{prop}
\begin{proof}
If $(u_n)_{n=1}^\infty$ is a Cauchy sequence in $\www$,
then from Remark \ref{rem-1} (i), we obtain $(u_n)_{n=1}^\infty$
converges to a function $u$ in $\vlp$. It remains to prove that $(u_n)_{n=1}^\infty$ converges
to $u$ also in $\www$. Passing to a subsequence if necessary we assume that
$$\lf\|u_{n+1}-u_{n}\r\|_{\www}<2^{-n},$$
and that $u_n\to u$ almost everywhere (see \cite[Proposition 2.67]{cruz03}). Moreover, by the definition \ref{de-4}, we know there exists
$g_n\in \vlp$ such that, for a.e. $x,y\in\cx$,
$$\lf|\lf(u_{n+1}-u_n\r)(x)-\lf(u_{n+1}-u_n\r)(y)\r|\leq\ \rho(x,y)^s\lf(g_n(x)+g_n(y)\r)$$
and  $$\lf\|g_n\r\|_{\vlp}<2^{-n}.$$
Moreover, we find that, for all $k\geq1$,
$$\lf|(u_{n+k}-u_n)(x)-(u_{n+k}-u_n)(y)\r|\leq\
\rho(x,y)^s\lf(\sum_{i=n}^{n+k-1}g_i(x)+\sum_{i=n}^{n+k-1}g_i(y)\r).$$
By taking $k\to \infty$, we obtain the inequality$$ |(u-u_n)(x)-(u-u_n)(y)|\leq\
\rho(x,y)^s\lf(\sum_{i=n}^{\infty}g_i(x)+\sum_{i=n}^{\infty}g_i(y)\r).$$
On the other hand, from the Fatou lemma (see \cite[Theorem 2.61]{cruz03}), we have
$$\lf\|\sum_{i=n}^{\infty}g_i\r\|_{\vlp}\leq
\sum_{i=n}^{\infty}\lf\|g_i\r\|_{\vlp}<\sum_{i=n}^{\infty}2^{-i}<2^{-n+1}.$$
Thus, we obtain $u\in \cw_{p(\cdot)}^s$ and  $u_n \to u$ in $\cw_{p(\cdot)}^s$. This
finishes the proof of Proposition \ref{prop-1}.
\end{proof}

\noindent\textbf{Acknowledgement}\quad
The authors would like to express their deep thanks to Professor Ciqiang Zhuo
for suggesting this interesting topic
and many useful discussion on the manuscript.

\section*{Statements and Declarations}
\noindent Funding: Not applicable.

\smallskip

\noindent Informed Consent Statement: Not applicable.

\smallskip

\noindent Data Availability Statement:  Not applicable.

\smallskip

\noindent Conflicts of Interest: The authors declares no conflict of interest.

\section*{Authors information}
\noindent Xiaosi Zhang

\smallskip

\noindent Key Laboratory of Computing and Stochastic Mathematics
(Ministry of Education), School of Mathematics and Statistics,
Hunan Normal University,
Changsha, Hunan 410081, People's Republic of China

\smallskip

\noindent {\it E-mail}: \texttt{xs102019@hunnu.edu.cn}

\medskip

\medskip

\noindent Qi Sun

\smallskip

\noindent Key Laboratory of Computing and Stochastic Mathematics
(Ministry of Education), School of Mathematics and Statistics,
Hunan Normal University,
Changsha, Hunan 410081, People's Republic of China

\smallskip

\noindent {\it E-mail}: \texttt{qsun7798@163.com; qsun7798@hunnu.edu.cn}

\end{document}